\documentclass{amsart}
\usepackage[dvips]{graphicx}
\usepackage{amssymb}

\newtheorem{theorem}{Theorem}[section]
\newtheorem{corollary}[theorem]{Corollary}

\newtheorem{proposition}[theorem]{Proposition}

\newtheorem{cl}{Claim}

\theoremstyle{remark}
\newtheorem{definition}[theorem]{\sc Definition}
\newtheorem{remark}[theorem]{\sc Remark}
\newtheorem{example}[theorem]{\sc Example}

\numberwithin{equation}{section}

\begin{document}

\title{Notes on toric Fano varieties associated to building sets}

\author{Yusuke Suyama}
\address{Department of Mathematics, Graduate School of Science, Osaka University,
1-1 Machikaneyama-cho, Toyonaka, Osaka 560-0043 JAPAN}
\email{y-suyama@cr.math.sci.osaka-u.ac.jp}


\keywords{toric geometry; Fano varieties; weak Fano varieties; building sets; nested sets;
graph associahedra; reflexive polytopes; graph cubeahedra; root systems.}

\date{\today}


\begin{abstract}
This article gives an overview of toric Fano and toric weak Fano varieties
associated to graphs and building sets.
We also study some properties of such toric Fano varieties
and discuss related topics.
\end{abstract}

\maketitle

\section{Introduction}

An $n$-dimensional {\it toric variety} is a normal algebraic variety $X$ over $\mathbb{C}$
containing the algebraic torus $(\mathbb{C}^*)^n$ as an open dense subset,
such that the natural action of $(\mathbb{C}^*)^n$ on itself extends to an action on $X$.
It is well known that the category of toric varieties is equivalent to the category of fans,
see \cite{oda88} for details.

A nonsingular projective algebraic variety is said to be {\it Fano}
if its anticanonical divisor is ample.
It is known that there are a finite number of isomorphism classes
of toric Fano varieties in any given dimension.
The classification of toric Fano varieties is a fundamental problem
and has been studied by many researchers.
In particular, {\O}bro \cite{obro07} gave an explicit algorithm
that classifies all toric Fano varieties for any dimension.
A nonsingular projective variety is said to be {\it weak Fano}
if its anticanonical divisor is nef and big.
Since the fan of a toric weak Fano variety determines a reflexive polytope,
there are a finite number of isomorphism classes
of toric weak Fano varieties of fixed dimension
as in the case of toric Fano varieties \cite{sato00}.
Sato \cite{sato02} classified toric weak Fano 3-folds
that are not Fano but are deformed to Fano under a small deformation,
which are called toric {\it weakened Fano} 3-folds.

There is a construction of nonsingular projective toric varieties from building sets,
which are formed by subsets of a finite set.
Since a finite simple graph determines a building set,
we can also associate to the graph a toric variety.
Toric varieties associated to building sets
were first studied by De Concini--Procesi \cite{deco95}
as smooth completions of hyperplane arrangement complements in a projective space
with normal crossing boundary divisor and they are now called {\it wonderful models}.
Building sets were originally defined as subspace arrangements
with some suitable properties.
In this article, we give an overview of toric Fano and toric weak Fano varieties
associated to finite simple graphs and building sets.
Furthermore, we provide some new results and discuss related topics.

This note is organized as follows:
In Section 2, we recall the construction of a toric variety from a building set
and a description of the intersection number of the anticanonical divisor
with a torus-invariant curve.
In Section 3, we survey characterizations of graphs and building sets
whose associated toric varieties are Fano or weak Fano.
In Section 4, we survey the relationship between the class of reflexive polytopes
determined by toric weak Fano varieties associated to building sets,
and that of reflexive polytopes associated to finite directed graphs.
Section 5 contains new results.
We study some properties of toric Fano varieties associated to building sets.
Finally, in Section 6, we discuss characterizations of graph cubeahedra and root systems
whose associated toric varieties are Fano or weak Fano.

\section{Building sets}

\subsection{Toric varieties associated to building sets}

\begin{definition}
A {\it building set} on a nonempty finite set $S$
is a finite set $B$ of nonempty subsets of $S$
satisfying the following conditions:
\begin{enumerate}
\item If $I, J \in B$ and $I \cap J \ne \emptyset$, then we have $I \cup J \in B$.
\item For every $i \in S$, we have $\{i\} \in B$.
\end{enumerate}
\end{definition}

We denote by $B_\mathrm{max}$
the set of all maximal (with respect to inclusion) elements of $B$.
An element of $B_\mathrm{max}$ is called a {\it $B$-component}
and $B$ is said to be {\it connected} if $B_\mathrm{max}=\{S\}$.
For a nonempty subset $C$ of $S$,
we call $B|_C=\{I \in B \mid I \subset C\}$ the {\it restriction} of $B$ to $C$.
The restriction $B|_C$ is a building set on $C$.
Note that $B|_C$ is connected if and only if $C \in B$.
For any building set $B$, we have $B=\bigsqcup_{C \in B_\mathrm{max}} B|_C$.
In particular, any building set is a disjoint union of connected building sets.

\begin{definition}\label{ys:nested}
Let $B$ be a building set.
A {\it nested set} of $B$ is a subset $N$ of $B \setminus B_\mathrm{max}$
satisfying the following conditions:
\begin{enumerate}
\item If $I, J \in N$, then we have either
$I \subset J$ or $J \subset I$ or $I \cap J=\emptyset$.
\item For any integer $k \geq 2$ and for any pairwise disjoint $I_1, \ldots, I_k \in N$,
the union $I_1 \cup \cdots \cup I_k$ is not in $B$.
\end{enumerate}
\end{definition}

The empty set is a nested set for any building set. 
The set $\mathcal{N}(B)$ of all nested sets of $B$ is called the {\it nested complex}.
The nested complex $\mathcal{N}(B)$
is in fact an abstract simplicial complex on $B \setminus B_\mathrm{max}$.
We denote by $\mathcal{N}(B)_\mathrm{max}$
the set of all maximal (with respect to inclusion) nested sets of $B$.

\begin{proposition}[{\cite[Proposition 4.1]{zele06}}]\label{ys:pure}
Let $B$ be a building set on $S$.
Then every maximal nested set of $B$ has the same cardinality $|S|-|B_\mathrm{max}|$.
In particular, the cardinality of every maximal nested set
of a connected building set on $S$ is $|S|-1$.
\end{proposition}

We construct a toric variety from a building set.
First, suppose that $B$ is a connected building set on $S=\{1, \ldots, n+1\}$.
We denote by $e_1, \ldots, e_n$ the standard basis for $\mathbb{R}^n$
and we put $e_{n+1}=-e_1-\cdots-e_n$.
For a nonempty subset $I$ of $S$,
we denote by $e_I$ the vector $\sum_{i \in I}e_i$ in $\mathbb{R}^n$.
Note that $e_S=0$.
For $N \in \mathcal{N}(B) \setminus \{\emptyset\}$,
we denote by $\mathbb{R}_{\geq 0}N$
the $|N|$-dimensional rational strongly convex polyhedral cone
$\sum_{I \in N}\mathbb{R}_{\geq 0}e_I$ in $\mathbb{R}^n$,
and we define $\mathbb{R}_{\geq 0}\emptyset$ to be $\{0\} \subset \mathbb{R}^n$.
Then $\Delta(B)=\{\mathbb{R}_{\geq 0}N \mid N \in \mathcal{N}(B)\}$
forms a fan in $\mathbb{R}^n$
and thus we have an $n$-dimensional toric variety $X(\Delta(B))$.
If $B$ is disconnected,
then we define $X(\Delta(B))=\prod_{C \in B_\mathrm{max}}X(\Delta(B|_C))$.

\begin{theorem}[{\cite[Corollary 5.2 and Theorem 6.1]{zele06}}]
Let $B$ be a building set.
Then the associated toric variety $X(\Delta(B))$ is nonsingular and projective.
\end{theorem}

\begin{example}\label{ys:example1}
Let $S=\{1, 2, 3\}$ and $B=\{\{1\}, \{2\}, \{3\}, \{2, 3\}, \{1, 2, 3\}\}$.
Then the nested complex $\mathcal{N}(B)$ is
\begin{align*}
&\{\emptyset, \{\{1\}\}, \{\{2\}\}, \{\{3\}\}, \{\{2, 3\}\},\\
&\{\{1\}, \{2\}\}, \{\{1\}, \{3\}\}, \{\{2\}, \{2, 3\}\}, \{\{3\}, \{2, 3\}\}\}.
\end{align*}
Hence we have the fan $\Delta(B)$ in Figure \ref{ys:fan1}.
Therefore the associated toric variety $X(\Delta(B))$
is $\mathbb{P}^2$ blown-up at one point.
\begin{figure}[h]
\begin{center}
\includegraphics[width=5cm]{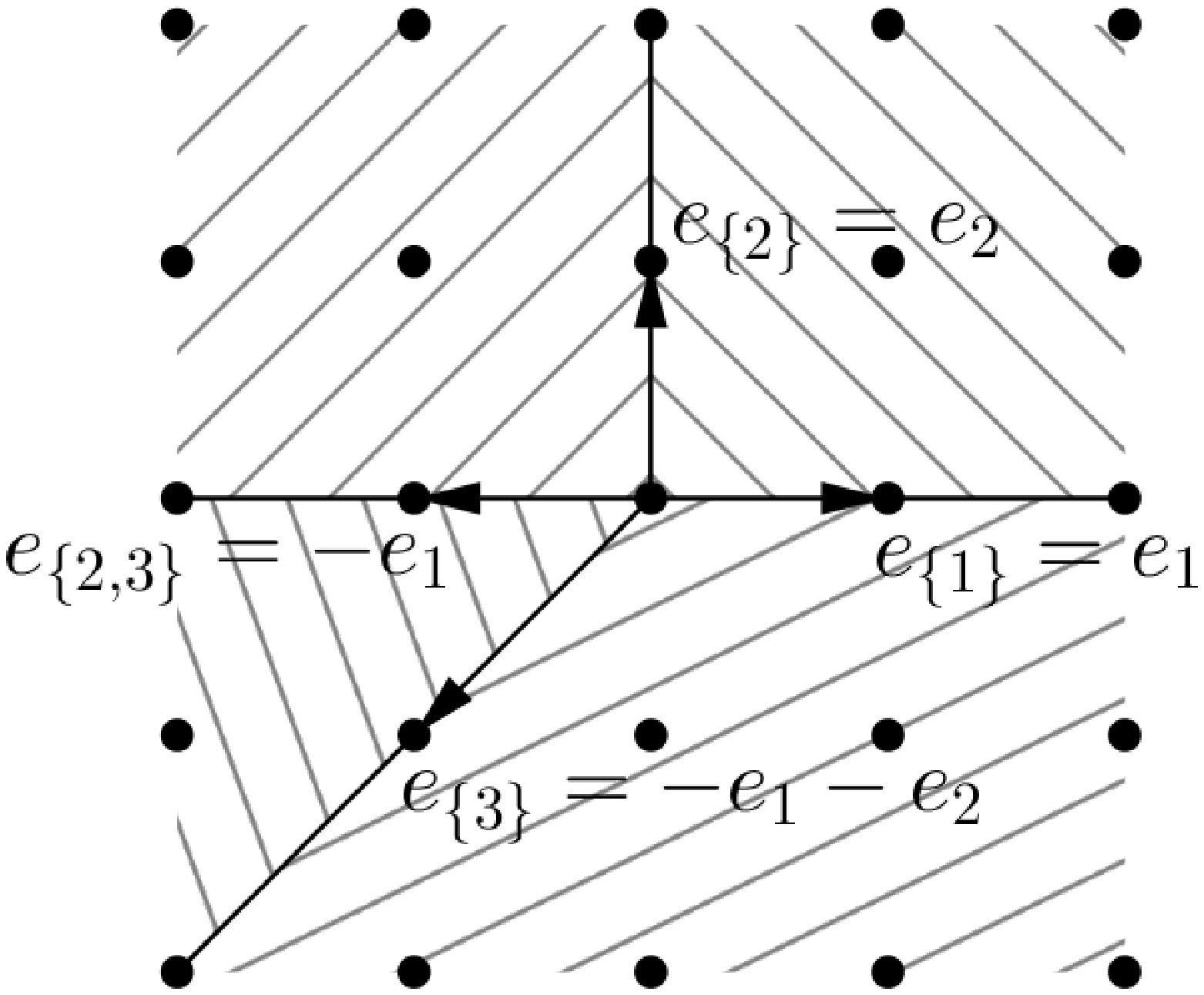}
\end{center}
\caption{the fan $\Delta(B)$.}
\label{ys:fan1}
\end{figure}
\end{example}

\begin{definition}
Let $G$ be a finite simple graph, that is,
a finite undirected graph with no loops and no multiple edges.
We denote by $V(G)$ and $E(G)$ its node set and edge set respectively.
For a subset $I$ of $V(G)$, the {\it induced subgraph} $G|_I$ is defined
by $V(G|_I)=I$ and $E(G|_I)=\{\{v, w\} \in E(G) \mid v, w \in I\}$.
The {\it graphical building set} $B(G)$ of $G$
is defined to be $\{I \subset V(G) \mid I \ne \emptyset, G|_I \mbox{ is connected}\}$.
The graphical building set $B(G)$ is in fact a building set on $V(G)$.
The graphical building set $B(G)$ is connected if and only if $G$ is connected.
\end{definition}

\begin{remark}
Carr--Devadoss \cite{carr06} introduced {\it graph associahedra} for finite simple graphs.
Graph associahedra include many important families of polytopes
such as associahedra (or Stasheff polytopes),
cyclohedra (or Bott--Taubes polytopes), stellohedra and permutohedra.
A graph associahedron can be realized as a smooth polytope
and $\Delta(B(G))$ coincides with the normal fan of the graph associahedron of $G$.
\end{remark}

\subsection{Intersection numbers}

Let $\Delta$ be a nonsingular complete fan in $\mathbb{R}^n$
and let $X(\Delta)$ be the associated toric variety.
We denote by $\Delta(r)$ the set of $r$-dimensional cones in $\Delta$ for $r=0, 1, \ldots, n$,
and denote by $V(\sigma)$ the orbit closure corresponding to $\sigma \in \Delta$.
The codimension of $V(\sigma)$ in $X(\Delta)$ equals the dimension of $\sigma$.

\begin{proposition}[e.g., \cite{oda88}]\label{ys:intersection number}
Let $\Delta$ be a nonsingular complete fan in $\mathbb{R}^n$
and $\tau=\mathbb{R}_{\geq 0}v_1+\cdots+\mathbb{R}_{\geq 0}v_{n-1} \in \Delta(n-1)$,
where $v_1, \ldots, v_{n-1}$ are primitive vectors in $\mathbb{Z}^n$.
Let $v$ and $v'$ be the distinct primitive vectors in $\mathbb{Z}^n$
such that $\tau+\mathbb{R}_{\geq 0}v, \tau+\mathbb{R}_{\geq 0}v' \in \Delta(n)$.
Then there exist unique integers $a_1, \ldots, a_{n-1}$
such that $v+v'+a_1v_1+\cdots+a_{n-1}v_{n-1}=0$.
Furthermore, the intersection number $(-K_{X(\Delta)}.V(\tau))$
of the anticanonical divisor $-K_{X(\Delta)}$ with the torus-invariant curve $V(\tau)$
equals $2+a_1+\cdots+a_{n-1}$.
\end{proposition}

\begin{proposition}\label{ys:Nakai}
Let $X(\Delta)$ be an $n$-dimensional nonsingular projective toric variety.
Then the following hold:
\begin{enumerate}
\item $X(\Delta)$ is Fano if and only if
the intersection number $(-K_{X(\Delta)}.V(\tau))$ is positive for every
$(n-1)$-dimensional cone $\tau$ in $\Delta$ (\cite[Lemma 2.20]{oda88}).
\item $X(\Delta)$ is weak Fano if and only if
$(-K_{X(\Delta)}.V(\tau))$ is nonnegative for every
$(n-1)$-dimensional cone $\tau$ in $\Delta$ (\cite[Proposition 6.17]{sato00}).
\end{enumerate}
\end{proposition}

To compute intersection numbers for the toric variety associated to a building set,
we need the following proposition:

\begin{proposition}[{\cite[Proposition 4.5]{zele06}}]\label{ys:pair}
Let $B$ be a building set on $S$
and let $I_1, I_2 \in B$ with $I_1 \ne I_2$ and $N \in \mathcal{N}(B)$
such that $N \cup \{I_1\}, N \cup \{I_2\} \in \mathcal{N}(B)_\mathrm{max}$.
Then the following hold:
\begin{enumerate}
\item We have $I_1 \not\subset I_2$ and $I_2 \not\subset I_1$.
\item If $I_1 \cap I_2 \ne \emptyset$, then $(B|_{I_1 \cap I_2})_\mathrm{max} \subset N$.
\item There exists a family $\{I_3, \ldots, I_k\} \subset N$ such that
$I_1 \cup I_2, I_3, \ldots, I_k$ are pairwise disjoint
and $I_1 \cup \cdots \cup I_k \in N \cup B_\mathrm{max}$
(the family $\{I_3, \ldots, I_k\}$ can be empty).
\end{enumerate}
\end{proposition}

Let $B$ be a connected building set on $S$
and let $I_1, I_2 \in B$ with $I_1 \ne I_2$ and $N \in \mathcal{N}(B)$
such that $N \cup \{I_1\}, N \cup \{I_2\} \in \mathcal{N}(B)_\mathrm{max}$.
Then by Proposition \ref{ys:pair} (3), there exists $\{I_3, \ldots, I_k\} \subset N$ such that
$I_1 \cup I_2, I_3, \ldots, I_k$ are pairwise disjoint
and $I_1 \cup \cdots \cup I_k \in N \cup B_\mathrm{max}=N \cup \{S\}$.
Since
\[
e_{I_1}+e_{I_2}-\sum_{C \in (B|_{I_1 \cap I_2})_\mathrm{max}}e_C
+e_3+\cdots+e_k-e_{I_1 \cup \cdots \cup I_k}=0,
\]
Proposition \ref{ys:intersection number} gives
\[
(-K_{X(\Delta(B))}.V(\mathbb{R}_{\geq0}N))=\left\{\begin{array}{ll}
k-|(B|_{I_1 \cap I_2})_\mathrm{max}|-1 & (I_1 \cup \cdots \cup I_k \in N), \\
k-|(B|_{I_1 \cap I_2})_\mathrm{max}| & (I_1 \cup \cdots \cup I_k=S). \end{array}\right.
\]
If $I_1 \cap I_2=\emptyset$,
then $(B|_{I_1 \cap I_2})_\mathrm{max}$ is understood to be empty.
The theorems in Section 3 are proved by using Proposition \ref{ys:Nakai}
and the equation above.

\section{Toric Fano varieties associated to building sets}

\subsection{Toric Fano varieties associated to finite simple graphs}

The following theorems characterize finite simple graphs
whose associated toric varieties are Fano or weak Fano.

\begin{theorem}[{\cite[Theorem 3.1]{suya16a}}]\label{ys:theorem1}
Let $G$ be a finite simple graph. Then the following are equivalent:
\begin{enumerate}
\item The associated nonsingular projective toric variety $X(\Delta(B(G)))$ is Fano.
\item Each connected component of $G$ has at most three nodes.
\end{enumerate}
\end{theorem}

\begin{theorem}[{\cite[Theorem 3.4]{suya16a}}]\label{ys:theorem2}
Let $G$ be a finite simple graph. Then the following are equivalent:
\begin{enumerate}
\item The toric variety $X(\Delta(B(G)))$ is weak Fano.
\item For any connected component $G'$ of $G$ and for any proper subset $I$ of $V(G')$,
the induced subgraph $G'|_I$ is not isomorphic to any of the following:
\begin{enumerate}
\item A cycle with $\geq 4$ nodes.
\item A diamond graph, that is, the graph obtained
by removing an edge from a complete graph with four nodes.
\end{enumerate}
\end{enumerate}
\end{theorem}

\begin{table}[htbp]
\begin{center}
\begin{tabular}{|c|c|c|}
\hline
\# of nodes & \# of connected graphs & \# of connected graphs whose \\
& & associated toric varieties are weak Fano \\
\hline
\hline
1 & \hphantom{00}1 & \hphantom{0}1 \\
\hline
2 & \hphantom{00}1 & \hphantom{0}1 \\
\hline
3 & \hphantom{00}2 & \hphantom{0}2 \\
\hline
4 & \hphantom{00}6 & \hphantom{0}6 \\
\hline
5 & \hphantom{0}21 & 10 \\
\hline
6 & 112 & 23 \\
\hline
\end{tabular}
\caption{the number of connected graphs
whose associated toric varieties are weak Fano.}
\label{ys:tbl1}
\end{center}
\end{table}

\begin{example}
\begin{enumerate}
\item Toric varieties associated to trees, cycles, complete graphs,
and a diamond graph are weak Fano.
\item The toric variety associated to the left graph in Figure \ref{ys:example2} is weak Fano,
but the toric variety associated to the right graph is not weak Fano.
\end{enumerate}
\end{example}

\begin{figure}[h]
\begin{center}
\includegraphics[width=3cm]{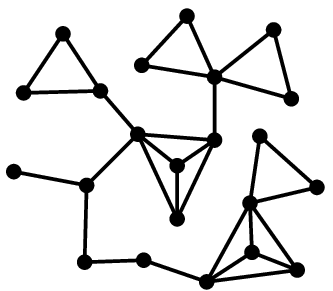}
\hspace*{1cm}
\includegraphics[width=1.8cm]{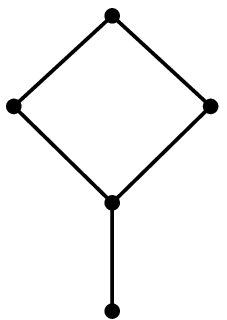}
\end{center}
\caption{examples.}
\label{ys:example2}
\end{figure}

\subsection{Toric Fano varieties associated to building sets}

\begin{theorem}[{\cite[Theorem 2.5]{suya16b}}]\label{ys:theorem3}
Let $B$ be a building set. Then the following are equivalent:
\begin{enumerate}
\item The toric variety $X(\Delta(B))$ is Fano.
\item For any $B$-component $C$ and for any $I_1, I_2 \in B|_C$
such that $I_1 \cap I_2 \ne \emptyset,
I_1 \not\subset I_2$ and $I_2 \not\subset I_1$,
we have $I_1 \cup I_2=C$ and $I_1 \cap I_2 \in B|_C$.
\end{enumerate}
\end{theorem}

\begin{table}[htbp]
\begin{center}
\begin{tabular}{|c|c|c|}
\hline
dimension & \# of toric Fano varieties & \# of toric Fano varieties \\
& & associated to building sets \\
\hline
\hline
1 & \hphantom{00}1 & \hphantom{00}1 \\
\hline
2 & \hphantom{00}5 & \hphantom{00}5 \\
\hline
3 & \hphantom{0}18 & \hphantom{0}14 \\
\hline
4 & 124 & \hphantom{0}50 \\
\hline
5 & 866 & 161 \\
\hline
\end{tabular}
\caption{the number of toric Fano varieties associated to building sets.}
\label{ys:tbl2}
\end{center}
\end{table}

\begin{example}\label{ys:examples}
If $|S| \leq 3$, then a connected building set $B$ on $S$
is isomorphic to one of the following six types:
\begin{enumerate}
\item $\{\{1\}\}$: a point, which is understood to be Fano.
\item $\{\{1\}, \{2\}, \{1, 2\}\}$: $\mathbb{P}^1$.
\item $\{\{1\}, \{2\}, \{3\}, \{1, 2, 3\}\}$: $\mathbb{P}^2$.
\item $\{\{1\}, \{2\}, \{3\}, \{1, 2\}, \{1, 2, 3\}\}$: $\mathbb{P}^2$ blown-up at one point.
\item $\{\{1\}, \{2\}, \{3\}, \{1, 2\}, \{1, 3\}, \{1, 2, 3\}\}$:
$\mathbb{P}^2$ blown-up at two points.
\item $\{\{1\}, \{2\}, \{3\}, \{1, 2\}, \{1, 3\}, \{2, 3\}, \{1, 2, 3\}\}$:
$\mathbb{P}^2$ blown-up at three non-collinear points.
\end{enumerate}
Thus $X(\Delta(B))$ is Fano in every case.
Since the disconnected building set $\{\{1\}, \{2\},\\ \{1, 2\}, \{3\}, \{4\}, \{3, 4\}\}$
yields $\mathbb{P}^1 \times \mathbb{P}^1$,
it follows that all toric Fano varieties of dimension $\leq 2$ are obtained from building sets.
\end{example}

\begin{theorem}[{\cite[Theorem 2.4]{suya17}}]\label{ys:theorem4}
Let $B$ be a building set. Then the following are equivalent:
\begin{enumerate}
\item The toric variety $X(\Delta(B))$ is weak Fano.
\item For any $B$-component $C$ and
for any $I_1, I_2 \in B|_C$ such that $I_1 \cap I_2 \ne \emptyset,
I_1 \not\subset I_2$ and $I_2 \not\subset I_1$,
we have at least one of the following:
\begin{enumerate}
\item $I_1 \cap I_2 \in B|_C$.
\item $I_1 \cup I_2=C$ and $|(B|_{I_1 \cap I_2})_\mathrm{max}| \leq 2$.
\end{enumerate}
\end{enumerate}
In particular, all toric varieties of dimension $\leq 3$
associated to building sets are weak Fano.
\end{theorem}

Theorem \ref{ys:theorem1} follows immediately from Theorem \ref{ys:theorem3}.
However, it is unclear
whether Theorem \ref{ys:theorem2} can be obtained from Theorem \ref{ys:theorem4}.

\section{Reflexive polytopes associated to building sets}

\subsection{Reflexive polytopes and directed graphs}

An $n$-dimensional integral convex polytope
$P$ in $\mathbb{R}^n$ is said to be {\it reflexive}
if the origin is in the interior of $P$ and the dual
$P^*=\{u \in \mathbb{R}^n \mid \langle u, v\rangle \geq -1 \mbox{ for any } v \in P\}$
is also an integral convex polytope,
where $\langle \cdot, \cdot\rangle$ denotes
the standard inner product on $\mathbb{R}^n$.
An $n$-dimensional integral convex polytope
in an $n$-dimensional real vector space with respect to a lattice
is said to be {\it smooth Fano} if the origin is the only lattice point
in the interior and the vertices of every facet form a basis for the lattice.
Note that any smooth Fano polytope in $\mathbb{R}^n$ is a reflexive polytope.

Let $\Delta$ be a nonsingular complete fan in $\mathbb{R}^n$.
If the associated toric variety $X(\Delta)$ is weak Fano,
then the convex hull of the primitive generators
of the rays in $\Delta(1)$ is a reflexive polytope.
This correspondence induces a bijection
between isomorphism classes of toric Fano varieties
and isomorphism classes of smooth Fano polytopes.
For a building set $B$ whose associated toric variety $X(\Delta(B))$
is weak Fano, we denote by $P_B$ the corresponding reflexive polytope.

Higashitani \cite{higa15} gave a construction
of integral convex polytopes from finite connected directed graphs
(with no loops and no multiple arrows).
Let $G$ be a finite connected directed graph
whose node set is $V(G)=\{1, \ldots, n+1\}$
and whose arrow set is $A(G) \subset V(G) \times V(G)$.
For $\overrightarrow{e}=(i, j) \in A(G)$,
we define $\rho(\overrightarrow{e}) \in \mathbb{R}^{n+1}$ to be $e_i-e_j$.
We define $P_G$ to be the convex hull of $\{\rho(\overrightarrow{e}) \mid
\overrightarrow{e} \in A(G)\}$ in $\mathbb{R}^{n+1}$.
Then $P_G$ is an integral convex polytope in the hyperplane
$H=\{(x_1, \ldots, x_{n+1}) \in \mathbb{R}^{n+1} \mid x_1+\cdots+x_{n+1}=0\}$.

Higashitani also characterized finite directed graphs
whose associated integral convex polytopes are smooth Fano,
see \cite{higa15} for details.

\subsection{Reflexive polytopes associated to building sets}

The following theorem is proved by using Theorem \ref{ys:theorem3} (2).

\begin{theorem}[{\cite[Theorem 4.2]{suya16b}}]\label{ys:theorem5}
Let $B$ be a building set.
If the associated toric variety $X(\Delta(B))$ is Fano,
then there exists a finite connected directed graph $G$ such that
$P_G \subset H$ is a smooth Fano polytope
(with respect to the lattice $H \cap \mathbb{Z}^{n+1}$)
and its associated fan is isomorphic to $\Delta(B)$,
that is, there is a linear isomorphism $F:H \rightarrow \mathbb{R}^n$
such that $F(H \cap \mathbb{Z}^{n+1})=\mathbb{Z}^n$ and $F(P_G)=P_B$,
where $n$ is the dimension of $X(\Delta(B))$.
\end{theorem}

\begin{remark}
Higashitani showed that any pseudo-symmetric smooth Fano polytope
is obtained from a finite directed graph (see \cite[Theorem 3.3]{higa15}).
On the other hand, there exists a building set $B$ such that $X(\Delta(B))$ is Fano
and the corresponding smooth Fano polytope is not pseudo-symmetric.
For example, the toric variety in Example \ref{ys:example1} is Fano
while the corresponding smooth Fano polytope is not pseudo-symmetric.
\end{remark}

\begin{example}
The converse of Theorem \ref{ys:theorem5} is not true.
The finite directed graphs in Figure \ref{ys:directedgraph2} yield smooth Fano 3-polytopes.
However, these polytopes cannot be obtained from building sets.
\begin{figure}[htbp]
\begin{center}
\includegraphics[width=4.5cm]{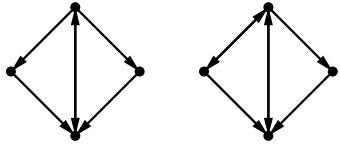}
\end{center}
\caption{directed graphs whose smooth Fano polytopes cannot be obtained from building sets.}
\label{ys:directedgraph2}
\end{figure}
\end{example}

\begin{remark}
In contrast to Theorem \ref{ys:theorem5},
there exist infinitely many reflexive polytopes associated to building sets
that cannot be obtained from finite directed graphs
(note that such reflexive polytopes are not smooth Fano).
On the other hand,
there also exists a reflexive polytope associated to a finite directed graph
that is not smooth Fano and cannot be obtained from any building set,
see \cite{suya17} for details.
\end{remark}

\section{Properties of toric Fano varieties associated to building sets}

\subsection{Extremal contractions}

Let $B$ be a connected building set on $S$
whose associated toric variety $X(\Delta(B))$ is Fano.
Let $I_1, I_2 \in B$ with $I_1 \ne I_2$ and $N \in \mathcal{N}(B)$
such that $N \cup \{I_1\}, N \cup \{I_2\} \in \mathcal{N}(B)_\mathrm{max}$
and $V(\mathbb{R}_{\geq0}N)$ is an extremal curve.
We find the wall relation associated to $V(\mathbb{R}_{\geq0}N)$.
We refer to \cite{reid83} for details on extremal contractions of toric varieties.

Suppose $I_1 \cap I_2=\emptyset$. By Proposition \ref{ys:pair} (3),
there exists $\{I_3, \ldots, I_k\} \subset N$ such that
$I_1 \cup I_2, I_3, \ldots, I_k$ are pairwise disjoint
and $I_1 \cup \cdots \cup I_k \in N \cup B_\mathrm{max}=N \cup \{S\}$.
If $I_1 \cup \cdots \cup I_k \in N$, then the wall relation is
$e_{I_1}+e_{I_2}+e_{I_3}+\cdots+e_{I_k}-e_{I_1 \cup \cdots \cup I_k}=0$.
If $I_1 \cup \cdots \cup I_k=S$, then the wall relation is
$e_{I_1}+e_{I_2}+e_{I_3}+\cdots+e_{I_k}=0$.

Suppose $I_1 \cap I_2 \ne \emptyset$. By Proposition \ref{ys:pair} (1),
we have $I_1 \not\subset I_2$ and $I_2 \not\subset I_1$.
Applying Theorem \ref{ys:theorem3} (2) to $I_1$ and $I_2$,
we have $I_1 \cup I_2=S$ and $I_1 \cap I_2 \in B$.
By Proposition \ref{ys:pair} (2),
we have $\{I_1 \cap I_2\}=(B|_{I_1 \cap I_2})_\mathrm{max} \subset N$.
Since $e_{I_1 \cup I_2}=e_S=0$, the wall relation is $e_{I_1}+e_{I_2}-e_{I_1 \cap I_2}=0$.

In every case, at most one of the coefficients
in the wall relation associated to $V(\mathbb{R}_{\geq0}N)$ is negative.
This implies the following corollary:

\begin{corollary}
Let $B$ be a building set.
If the associated toric variety $X(\Delta(B))$ is Fano,
then there are no small extremal contractions of $X(\Delta(B))$.
\end{corollary}

\subsection{Toric 2-Fano varieties associated to building sets}

2-Fano varieties were introduced by de Jong--Starr \cite{dejo07}.

\begin{definition}
A Fano variety $X$ is said to be {\bf 2-Fano}
if the second Chern character $\mathrm{ch}_2(X)=(c_1(X)^2-2c_2(X))/2$ is ample,
that is, the intersection number $(\mathrm{ch}_2(X). S)$
is positive for any surface $S$ on $X$.
\end{definition}

At present, the only known examples of toric 2-Fano varieties are projective spaces.
In this subsection, we prove the following theorem:

\begin{theorem}\label{ys:theorem6}
Let $B$ be a building set whose associated toric variety $X(\Delta(B))$ is Fano.
Then $X(\Delta(B))$ is 2-Fano if and only if it is a projective space.
\end{theorem}

Let $\Delta$ be a nonsingular complete fan in $\mathbb{R}^n$
and $v_1, \ldots, v_m$ be the primitive generators of the rays in $\Delta(1)$.
We put $D_i=V(\mathbb{R}_{\geq0}v_i)$ for $i=1, \ldots, m$.
For a torus-invariant subvariety $Y \subset X(\Delta)$ of codimension $l$, we put
\[
I_{Y/X(\Delta)}=\sum_{(i_1, \ldots, i_l) \in \{1, \ldots, m\}^l}
(D_{i_1} \cdots D_{i_l}.Y)X_{i_1} \cdots X_{i_l}
\in \mathbb{Z}[X_1, \ldots, X_m].
\]

\begin{example}
Let $v_{i_1}, \ldots, v_{i_{n-1}}, v_{j_1}, v_{j_2}$ be distinct primitive generators
such that
$\tau=\mathbb{R}_{\geq0}v_{i_1}+\cdots+\mathbb{R}_{\geq0}v_{i_{n-1}} \in \Delta(n-1)$
and $\tau+\mathbb{R}_{\geq0}v_{j_1}, \tau+\mathbb{R}_{\geq0}v_{j_2} \in \Delta(n)$.
If $v_{j_1}+v_{j_2}+a_1v_{i_1}+\cdots+a_{n-1}v_{i_{n-1}}=0$ for $a_1, \ldots, a_{n-1} \in \mathbb{Z}$,
then $I_{V(\tau)/X(\Delta)}=X_{j_1}+X_{j_2}+a_1X_{i_1}+\cdots+a_{n-1}X_{i_{n-1}}$.
\end{example}

\begin{theorem}[{\cite[Theorem 3.5]{sato12}}]\label{ys:Hirzebruch}
Let $X(\Delta)$ be a nonsingular projective toric variety
and let $S$ be a torus-invariant surface on $X(\Delta)$
isomorphic to a Hirzebruch surface of degree $a \geq 0$.
Then $I_{S/X(\Delta)}=aI_{C_\mathrm{fib}/X(\Delta)}^2
+2I_{C_\mathrm{fib}/X(\Delta)}I_{C_\mathrm{neg}/X(\Delta)}$,
where $C_\mathrm{fib}$ is a fiber of the projection $S \rightarrow \mathbb{P}^1$
and $C_\mathrm{neg}$ is the negative section of $S$.
\end{theorem}

For a nonsingular complete toric variety $X(\Delta)$,
it is known that $\mathrm{ch}_2(X(\Delta))=(D_1^2+\cdots+D_m^2)/2$.
Hence if $I_{S/X(\Delta)}$ has an expression of the form
$\sum_{1 \leq i \leq j \leq m}a_{ij}X_iX_j$,
then we have $(\mathrm{ch}_2(X(\Delta)). S)=(a_{11}+\cdots+a_{mm})/2$.

\begin{proof}[Proof of Theorem \ref{ys:theorem6}]
We only need to prove the necessity.
Suppose that $X(\Delta(B))$ is not a projective space.
If $X(\Delta(B))$ is a surface,
then Noether's theorem implies that $\mathrm{ch}_2(X(\Delta(B)))$ is not ample.
Since a product of nonsingular projective varieties
is not 2-Fano (see \cite[Lemma 17]{arau13}),
we may assume that $n \geq 3$
and $B$ is a connected building set on $S=\{1, \ldots, n+1\}$.
By the assumption,
$B'=B \setminus \{\{1\}, \ldots, \{n+1\}, S\}$ is nonempty.
For a torus-invariant subvariety $Y$ of $X(\Delta(B))$,
we regard $I_{Y/X(\Delta(B))}$ as a polynomial
in the polynomial ring $\mathbb{Z}[X_I \mid I \in B \setminus \{S\}]$.

{\it Case 1}. Suppose that every maximal element of $B'$ has cardinality $n$.
We pick a maximal element $I$ of $B'$.
We may assume $I=\{1, \ldots, n\}$.

{\it Subcase 1.1}. Suppose that there is no element of $B'$ containing $n+1$.
Let $K_1, \ldots, K_r$ be all maximal elements of $B|_I \setminus \{I\}$.
Note that $r \geq 2$.
Since $X(\Delta(B))$ is Fano,
Theorem \ref{ys:theorem3} (2) implies that $I$ is the disjoint union of $K_1, \ldots, K_r$.
We pick $N_i \in \mathcal{N}(B|_{K_i})_\mathrm{max}$ for each $i=1, \ldots, r$.
Let
\[
N=N_1 \cup \cdots \cup N_r \cup \{K_3, \ldots, K_r\}.
\]

\begin{cl}
The families
$N \cup \{I, K_1\}, N \cup \{K_1, \{n+1\}\}, N \cup \{\{n+1\}, K_2\}, N \cup \{K_2, I\}$
are maximal nested sets of $B$.
\end{cl}

\begin{proof}
For each family, condition (1) of Definition \ref{ys:nested} is obviously satisfied.
We check condition (2).
Suppose that $I_1, \ldots, I_s \in N_1 \cup \cdots \cup N_r \cup \{K_1, K_3, \ldots, K_r\}$
such that $I_1, \ldots, I_s$ are pairwise disjoint.
We have $I_1 \cup \cdots \cup I_s \subsetneq I$,
since if $I_1 \cup \cdots \cup I_s=I$
then $K_2$ is the disjoint union of $I_{i_1}, \ldots, I_{i_t}$
for some $1 \leq i_1 < \cdots < i_t \leq s$ with $\{I_{i_1}, \ldots, I_{i_t}\} \subset N_2$,
which contradicts that $N_2$ is a nested set.

Assume that $N \cup \{I, K_1\}$ does not satisfy condition (2) for contradiction.
Since $I$ is a unique maximal element of $N \cup \{I, K_1\}$,
we may assume that $s \geq 2$ and $I_1 \cup \cdots \cup I_s \in B$.
If $K_i=I_k$ for some $i, k$,
then $K_i \subsetneq I_1 \cup \cdots \cup I_s \in B|_I \setminus \{I\}$,
a contradiction.
If $I_1, \ldots, I_s \in N_1 \cup \cdots \cup N_r$,
then $K_i \supset I_1$ for some $i$
and thus $K_i \subsetneq I_1 \cup \cdots \cup I_s \cup K_i \in B|_I \setminus \{I\}$,
a contradiction.
Therefore $N \cup \{I, K_1\}$ is a nested set.

We show that $N \cup \{K_1, \{n+1\}\}$ satisfies condition (2).
If $s \geq 2$ and $I_1 \cup \cdots \cup I_s \in B$,
then a similar argument leads to a contradiction.
If $I_1 \cup \cdots \cup I_s \cup \{n+1\} \in B$,
then it contradicts that there is no element of $B'$ containing $n+1$.
Therefore $N \cup \{K_1, \{n+1\}\}$ is a nested set.

Similarly, $N \cup \{\{n+1\}, K_2\}$ and $N \cup \{K_2, I\}$ are nested sets.
Each nested set is maximal since it has cardinality $n$.
This completes the proof of the claim.
\end{proof}

Since $e_{K_1}+e_{K_2}+e_{K_3}+\cdots+e_{K_r}-e_I=0$ and $e_I+e_{\{n+1\}}=0$,
the surface $V(\mathbb{R}_{\geq0}N)$ is isomorphic to a Hirzebruch surface of degree one
and $C_\mathrm{neg}=V(\mathbb{R}_{\geq0}(N \cup \{I\}))$.
We choose $C_\mathrm{fib}=V(\mathbb{R}_{\geq0}(N \cup \{K_1\}))$.
Since
\begin{align*}
I_{C_\mathrm{neg}/X(\Delta(B))}&=X_{K_1}+X_{K_2}+X_{K_3}+\cdots+X_{K_r}-X_I,\\
I_{C_\mathrm{fib}/X(\Delta(B))}&=X_I+X_{\{n+1\}},\\
I_{V(\mathbb{R}_{\geq0}N)/X(\Delta(B))}
&=(X_I+X_{\{n+1\}})^2+2(X_I+X_{\{n+1\}})(X_{K_1}+\cdots+X_{K_r}-X_I),
\end{align*}
we have
$(\mathrm{ch}_2(X(\Delta(B))).V(\mathbb{R}_{\geq0}N))=(1+1-2)/2=0$.

{\it Subcase 1.2}. Suppose that there exists an element of $B'$ containing $n+1$.
We pick a maximal element $J$ of $B'$ containing $n+1$.
We may assume $J=\{1, \ldots, n-1, n+1\}$.
Since $X(\Delta(B))$ is Fano,
Theorem \ref{ys:theorem3} (2) implies $\{1, \ldots, n-1\}=I \cap J \in B$.
Let $K_1, \ldots, K_r$ be all maximal elements of $B|_{I \cap J} \setminus \{I \cap J\}$.
Note that $r \geq 2$.
Theorem \ref{ys:theorem3} (2) implies that $I \cap J$ is the disjoint union of $K_1, \ldots, K_r$.
We pick $N_i \in \mathcal{N}(B|_{K_i})_\mathrm{max}$ for each $i=1, \ldots, r$.
Let
\[
N=N_1 \cup \cdots \cup N_r \cup \{K_3, \ldots, K_r, \{n\}\}.
\]

\begin{cl}
The families
$N \cup \{I, K_1\}, N \cup \{K_1, \{n+1\}\}, N \cup \{\{n+1\}, K_2\}, N \cup \{K_2, I\}$
are maximal nested sets of $B$.
\end{cl}

\begin{proof}
It suffices to show that $N \cup \{I, K_1\}$ and $N \cup \{K_1, \{n+1\}\}$ satisfy
condition (2) of Definition \ref{ys:nested}.
Suppose that $I_1, \ldots, I_s \in N_1 \cup \cdots \cup N_r \cup \{K_1, K_3, \ldots, K_r\}$
such that $I_1, \ldots, I_s$ are pairwise disjoint.
Note that $I_1 \cup \cdots \cup I_s \subsetneq I \cap J$.

We show that $N \cup \{I, K_1\}$ satisfies condition (2).
If $s \geq 2$ and $I_1 \cup \cdots \cup I_s \in B$,
then an argument similar to the proof of Claim 1 leads to a contradiction.
Suppose $I_1 \cup \cdots \cup I_s \cup \{n\} \in B$.
Applying Theorem \ref{ys:theorem3} (2)
to $I_1 \cup \cdots \cup I_s \cup \{n\}$ and $I \cap J$,
we have $S=(I_1 \cup \cdots \cup I_s \cup \{n\}) \cup (I \cap J)=I$, a contradiction.
Since $I$ is a unique maximal element of $N \cup \{I, K_1\}$,
it follows that $N \cup \{I, K_1\}$ is a nested set.

We show that $N \cup \{K_1, \{n+1\}\}$ satisfies condition (2).
If $I_1 \cup \cdots \cup I_s \in B$ with $s \geq 2$
or $I_1 \cup \cdots \cup I_s \cup \{n\} \in B$,
then a similar argument leads to a contradiction.
Suppose $I_1 \cup \cdots \cup I_s \cup \{n+1\} \in B$.
Applying Theorem \ref{ys:theorem3} (2)
to $I_1 \cup \cdots \cup I_s \cup \{n+1\}$ and $I \cap J$,
we have $S=(I_1 \cup \cdots \cup I_s \cup \{n+1\}) \cup (I \cap J)=J$, a contradiction.
Suppose that $\{n, n+1\} \in B$ or $I_1 \cup \cdots \cup I_s \cup \{n, n+1\} \in B$.
Then there exists a maximal element $L$ of $B'$ containing it.
Theorem \ref{ys:theorem3} (2) implies $I \cap L, J \cap L \in B$.
Since $n \geq 3$, we have $(I \cap L) \cap (J \cap L)=(I \cap J) \cap L \ne \emptyset$
and thus Theorem \ref{ys:theorem3} (2) implies
$S=(I \cap L) \cup (J \cap L) \subset L$, a contradiction.
Therefore $N \cup \{K_1, \{n+1\}\}$ is a nested set.
\end{proof}

Since $e_{K_1}+e_{K_2}+e_{K_3}+\cdots+e_{K_r}+e_{\{n\}}-e_I=0$ and $e_I+e_{\{n+1\}}=0$,
the surface $V(\mathbb{R}_{\geq0}N)$ is isomorphic to a Hirzebruch surface of degree one
and $C_\mathrm{neg}=V(\mathbb{R}_{\geq0}(N \cup \{I\}))$.
We choose $C_\mathrm{fib}=V(\mathbb{R}_{\geq0}(N \cup \{K_1\}))$.
Since
\begin{align*}
I_{C_\mathrm{neg}/X(\Delta(B))}&=X_{K_1}+X_{K_2}+X_{K_3}+\cdots+X_{K_r}+X_{\{n\}}-X_I,\\
I_{C_\mathrm{fib}/X(\Delta(B))}&=X_I+X_{\{n+1\}},\\
I_{V(\mathbb{R}_{\geq0}N)/X(\Delta(B))}
&=(X_I+X_{\{n+1\}})^2\\
&+2(X_I+X_{\{n+1\}})(X_{K_1}+\cdots+X_{K_r}+X_{\{n\}}-X_I),
\end{align*}
we have
$(\mathrm{ch}_2(X(\Delta(B))).V(\mathbb{R}_{\geq0}N))=(1+1-2)/2=0$.

{\it Case 2}. Suppose that there exists a maximal element $I$ of $B'$ such that $|I| \leq n-1$.
Let $K_1, \ldots, K_r$ be all maximal elements of $B|_I \setminus \{I\}$
and $L_1, \ldots, L_s$ be all maximal elements of
$B|_{S \setminus I} \setminus \{S \setminus I\}$.
Note that $r, s \geq 2$.
Since $X(\Delta(B))$ is Fano,
Theorem \ref{ys:theorem3} (2) implies that $I$ is the disjoint union of $K_1, \ldots, K_r$
and $S \setminus I$ is the disjoint union of $L_1, \ldots, L_s$.
We pick $N_i \in \mathcal{N}(B|_{K_i})_\mathrm{max}$ for each $i=1, \ldots, r$ and
$N'_j \in \mathcal{N}(B|_{L_j})_\mathrm{max}$ for each $j=1, \ldots, s$.
Let
\[
N=N_1 \cup \cdots \cup N_r \cup N'_1 \cup \cdots \cup N'_s
\cup \{K_3, \ldots, K_r, L_3, \ldots, L_s, I\}.
\]

\begin{cl}
The families
$N \cup \{K_1, L_1\}, N \cup \{L_1, K_2\}, N \cup \{K_2, L_2\}, N \cup \{L_2, K_1\}$
are maximal nested sets of $B$.
\end{cl}

\begin{proof}
It suffices to show that $N \cup \{K_1, L_1\}$
satisfies condition (2) of Definition \ref{ys:nested}.
Suppose that $I_1, \ldots, I_t \in N_1 \cup \cdots \cup N_r \cup \{K_1, K_3, \ldots, K_r, I\}$
and $J_1, \ldots, J_u \in N'_1 \cup \cdots \cup N'_s \cup \{L_1, L_3, \ldots, L_s\}$
such that $I_1, \ldots, I_t, J_1, \ldots, J_u$ are pairwise disjoint.
Note that $J_1 \cup \cdots \cup J_u \subsetneq S \setminus I$.

If $t \geq 2$ and $I_1 \cup \cdots \cup I_t \in B$,
then an argument similar to the proof of Claim 1 leads to a contradiction.
If $I_1 \cup \cdots \cup I_t \cup J_1 \cup \cdots \cup J_u \in B$,
then we have $I \subsetneq I \cup J_1 \cup \cdots \cup J_u \in B'$,
which contradicts that $I$ is a maximal element of $B'$.
Suppose that $u \geq 2$ and $J_1 \cup \cdots \cup J_u \in B$.
If $L_j=J_k$ for some $j, k$,
then $L_j \subsetneq J_1 \cup \cdots \cup J_u
\in B|_{S \setminus I} \setminus \{S \setminus I\}$,
a contradiction.
If $J_1, \ldots, J_u \in N'_1 \cup \cdots \cup N'_s$,
then $L_j \supset J_1$ for some $j$
and thus $L_j \subsetneq J_1 \cup \cdots \cup J_u \cup L_j
\in B|_{S \setminus I} \setminus \{S \setminus I\}$,
a contradiction.
Therefore $N \cup \{K_1, L_1\}$ is a nested set.
\end{proof}

Since $e_{K_1}+e_{K_2}+e_{K_3}+\cdots+e_{K_r}-e_I=0$
and $e_{L_1}+e_{L_2}+e_{L_3}+\cdots+e_{L_s}+e_I=0$,
the surface $V(\mathbb{R}_{\geq0}N)$ is isomorphic to $\mathbb{P}^1 \times \mathbb{P}^1$.
We choose $C_\mathrm{neg}=V(\mathbb{R}_{\geq0}(N \cup \{L_1\}))$
and $C_\mathrm{fib}=V(\mathbb{R}_{\geq0}(N \cup \{K_1\}))$. Since
\begin{align*}
I_{C_\mathrm{neg}/X(\Delta(B))}&=X_{K_1}+X_{K_2}+X_{K_3}+\cdots+X_{K_r}-X_I,\\
I_{C_\mathrm{fib}/X(\Delta(B))}&=X_{L_1}+X_{L_2}+X_{L_3}+\cdots+X_{L_s}+X_I,\\
I_{V(\mathbb{R}_{\geq0}N)/X(\Delta(B))}
&=2(X_{L_1}+\cdots+X_{L_s}+X_I)(X_{K_1}+\cdots+X_{K_r}-X_I),
\end{align*}
we have
$(\mathrm{ch}_2(X(\Delta(B))).V(\mathbb{R}_{\geq0}N))=(-2)/2=-1$.

In every case, $\mathrm{ch}_2(X(\Delta(B)))$ is not ample.
This completes the proof of the theorem.
\end{proof}

\section{Related topics}

\subsection{Toric Fano varieties associated to graph cubeahedra}

Devadoss--Heath--Vipismakul \cite{deva11} introduced {\it graph cubeahedra}
for finite simple graphs
and proved that graph associahedra and graph cubeahedra
appear as some compactified moduli spaces of marked bordered Riemann surfaces.
A graph cubeahedron can also be realized as a smooth polytope
and thus we can obtain a nonsingular projective toric variety.
Let $G$ be a finite simple graph on $V(G)=\{1, \ldots, n\}$,
and let $\Box^n$ be an $n$-cube whose facets are labeled by
$1, \ldots, n$ and $\overline{1}, \ldots, \overline{n}$,
where the two facets labeled by $i$ and $\overline{i}$ are on opposite sides.
Then every face of $\Box^n$ is labeled by a subset
$I$ of $\{1, \ldots, n, \overline{1}, \ldots, \overline{n}\}$ such that
$I \cap \{1, \ldots, n\}$ and $\{i \in \{1, \ldots, n\} \mid \overline{i} \in I\}$ are disjoint,
that is, the face corresponding to $I$ is the intersection of the facets
labeled by the elements of $I$.
The {\it graph cubeahedron} $\Box_G$ is obtained from $\Box^n$
by truncating the faces labeled by the elements of the graphical building set $B(G)$
in increasing order of dimension.
We have a bijection between the set of facets of $\Box_G$
and $B(G) \cup \{\{\overline{1}\}, \ldots, \{\overline{n}\}\}$.
For $I \in B(G) \cup \{\{\overline{1}\}, \ldots, \{\overline{n}\}\}$,
we denote by $F_I$ the corresponding facet.

\begin{theorem}[{\cite[Theorem 12]{deva11}}]
Let $G$ be a finite simple graph.
Then the two facets $F_I$ and $F_J$ of the graph cubeahedron $\Box_G$ intersect
if and only if one of the following holds:
\begin{enumerate}
\item $I, J \in B(G)$ and we have either
$I \subset J$ or $J \subset I$ or $I \cup J \notin B(G)$.
\item One of $I$ and $J$, say $I$, is in $B(G)$ and $J=\{\overline{j}\}$
for some $j \in \{1, \ldots, n\} \setminus I$.
\item $I=\{\overline{i}\}$ and $J=\{\overline{j}\}$ for some $i, j \in \{1, \ldots, n\}$.
\end{enumerate}
Furthermore, $\Box_G$ is a flag polytope.
\end{theorem}

We can realize $\Box_G$ as a smooth polytope
such that the outward-pointing primitive normal vector $e_I$ of the facet $F_I$ is
\[
\left\{\begin{array}{ll}
\sum_{i \in I}e_i & (I \in B(G)), \\
-e_i & (I=\{\overline{i}\}, i \in \{1, \ldots, n\}), \end{array}\right.
\]
for any $I \in B(G) \cup \{\{\overline{1}\}, \ldots, \{\overline{n}\}\}$.
Let
\[
\mathcal{N}^\Box(G)=\{N
\subset B(G) \cup \{\{\overline{1}\}, \ldots, \{\overline{n}\}\} \mid
F_I \mbox{ and } F_J \mbox{ intersect for any } I, J \in N\}.
\]
For $N \in \mathcal{N}^\Box(G) \setminus \{\emptyset\}$,
we denote by $\sigma_N$ the $|N|$-dimensional cone
$\sum_{I \in N}\mathbb{R}_{\geq 0}e_I$ in $\mathbb{R}^n$,
and we define $\sigma_\emptyset$ to be $\{0\} \subset \mathbb{R}^n$.
Then $\{\sigma_N \mid N \in \mathcal{N}^\Box(G)\}$
coincides with the normal fan $\Delta_{\Box_G}$.
Note that if $G_1, \ldots, G_m$ are the connected components of $G$,
then $X(\Delta_{\Box_G})$ is isomorphic to the product
$X(\Delta_{\Box_{G_1}}) \times \cdots \times X(\Delta_{\Box_{G_m}})$.

\begin{remark}
Manneville--Pilaud proved
that for connected graphs $G$ and $G'$,
the graph associahedron of $G$ and the graph cubeahedron of $G'$
are combinatorially equivalent
if and only if $G$ is a tree with at most one node of degree more than two
and $G'$ is its line graph (see \cite[Proposition 64]{mann17}).
Hence we can see that there exist many toric varieties associated to graph cubeahedra
that cannot be obtained from graph associahedra.
\end{remark}

Recently, the author proved the following:

\begin{theorem}[{\cite[Theorem 5]{suya18}}]\label{ys:theorem7}
Let $G$ be a finite simple graph. Then the following are equivalent:
\begin{enumerate}
\item The nonsingular projective toric variety $X(\Delta_{\Box_G})$
associated to the graph cubeahedron $\Box_G$ is Fano.
\item Each connected component of $G$ has at most two nodes.
\end{enumerate}
In particular, any toric Fano variety associated to a graph cubeahedron
is a product of copies of $\mathbb{P}^1$
and $\mathbb{P}^1 \times \mathbb{P}^1$ blown-up at one point.
\end{theorem}

\begin{theorem}[{\cite[Theorem 6]{suya18}}]\label{ys:theorem8}
Let $G$ be a finite simple graph. Then the following are equivalent:
\begin{enumerate}
\item The toric variety $X(\Delta_{\Box_G})$ is weak Fano.
\item For any subset $I$ of $V(G)$,
the induced subgraph $G|_I$ is not isomorphic to any of the following:
\begin{enumerate}
\item A cycle with $\geq 4$ nodes.
\item A diamond graph.
\item A claw, that is, a star with three edges.
\end{enumerate}
\end{enumerate}
In particular, if $X(\Delta_{\Box_G})$ is weak Fano, then $G$ is chordal.
\end{theorem}

\begin{table}[htbp]
\begin{center}
\begin{tabular}{|c|c|c|}
\hline
\# of nodes & \# of connected & \# of connected graphs whose graph cubeahedra \\
& graphs & correspond to toric weak Fano varieties \\
\hline
\hline
1 & \hphantom{00}1 & \hphantom{0}1 \\
\hline
2 & \hphantom{00}1 & \hphantom{0}1 \\
\hline
3 & \hphantom{00}2 & \hphantom{0}2 \\
\hline
4 & \hphantom{00}6 & \hphantom{0}3 \\
\hline
5 & \hphantom{0}21 & \hphantom{0}6 \\
\hline
6 & 112 & 11 \\
\hline
\end{tabular}
\caption{the number of connected graphs
whose graph cubeahedra correspond to toric weak Fano varieties.}
\label{ys:tbl3}
\end{center}
\end{table}

\begin{remark}
A connected graph satisfies condition (2) of Theorem \ref{ys:theorem8}
if and only if it is the line graph of a tree (see, for example \cite[Theorem 8.5]{hara69}),
and such a graph is called a {\it claw-free block graph}.
Therefore for a connected graph $G$, the toric variety $X(\Delta_{\Box_G})$ is weak Fano
if and only if $G$ is the line graph of a tree.
\end{remark}

\subsection{Toric Fano varieties associated to Weyl chambers}

The collection of Weyl chambers for a root system determines a nonsingular complete fan.
Let $\Phi$ be a root system in a Euclidean space $V$.
Let $M(\Phi)$ be the root lattice $\sum_{\alpha \in \Phi}\mathbb{Z}\alpha$
and $N(\Phi)$ be its dual lattice $\mathrm{Hom}_\mathbb{Z}(M(\Phi), \mathbb{Z})$.
For a set of simple roots $S \subset \Phi$,
we define a cone $\sigma_S=\{v \in N(\Phi)_\mathbb{R} \mid
\langle u, v \rangle \geq 0 \mbox{ for any }u \in S\}$,
where $N(\Phi)_\mathbb{R}=N(\Phi) \otimes_\mathbb{Z} \mathbb{R}$.
The set $\Delta(\Phi)$ of all such cones and their faces
forms a nonsingular complete fan in the lattice $N(\Phi)$.
The associated toric variety $X(\Delta(\Phi))$ is nonsingular and projective \cite{baty11}.

\begin{proposition}\label{ys:Cartan}
Let $S=\{\alpha_1, \ldots, \alpha_n\} \subset \Phi$ be a set of simple roots
and $\omega_1, \ldots, \omega_n$ be the dual basis of $\alpha_1, \ldots, \alpha_n$.
We put $a_{ij}=2(\alpha_i, \alpha_j)/(\alpha_j, \alpha_j)$.
Then we have
$(-K_{X(\Delta(\Phi))}.V(\sum_{i \ne j}\mathbb{R}_{\geq0}\omega_i))=\sum_{i=1}^na_{ij}$
for any $j=1, \ldots, n$.
\end{proposition}

\begin{proof}
For $\alpha \in \Phi$, we denote by $s_\alpha:V \rightarrow V$
the reflection through the hyperplane associated to $\alpha$,
that is, the map defined by
$s_\alpha(\lambda)=\lambda-(2(\lambda, \alpha)/(\alpha, \alpha))\alpha$.
For $j=1, \ldots, n$, we consider the set of simple roots
$s_{\alpha_j}S=\{s_{\alpha_j}(\alpha_1), \ldots, s_{\alpha_j}(\alpha_n)\}$.
The dual basis of $s_{\alpha_j}(\alpha_1), \ldots, s_{\alpha_j}(\alpha_n)$
is $\omega_1s_{\alpha_j}, \ldots, \omega_ns_{\alpha_j}$.
We see that $\omega_is_{\alpha_j}=\omega_i$ on $M(\Phi)$ for any $i \ne j$.
On the other hand, we have
\[
(\omega_js_{\alpha_j})(\alpha_i)
=\omega_j\left(\alpha_i-\frac{2(\alpha_i, \alpha_j)}{(\alpha_j, \alpha_j)}\alpha_j\right)
=\left\{\begin{array}{ll}
-1 & (i=j), \\
-a_{ij} & (i \ne j). \end{array}\right.
\]
Thus $\omega_js_{\alpha_j}=-\omega_j-\sum_{i \ne j}a_{ij}\omega_i$ on $M(\Phi)$.
Hence we have
\[
\sigma_{s_{\alpha_j}S}=\sum_{i=1}^n\mathbb{R}_{\geq0}(\omega_is_{\alpha_j})
=\sum_{i \ne j}\mathbb{R}_{\geq0}\omega_i
+\mathbb{R}_{\geq0}(-\omega_j-\sum_{i \ne j}a_{ij}\omega_i)
\]
and $\sigma_S \cap \sigma_{s_{\alpha_j}S}=\sum_{i \ne j}\mathbb{R}_{\geq0}\omega_i$.
Since $\omega_j+(-\omega_j-\sum_{i \ne j}a_{ij}\omega_i)+\sum_{i \ne j}a_{ij}\omega_i=0$,
Proposition \ref{ys:intersection number} gives
\[
(-K_{X(\Delta(\Phi))}.V(\sum_{i \ne j}\mathbb{R}_{\geq0}\omega_i))
=2+\sum_{i \ne j}a_{ij}=\sum_{i=1}^na_{ij}.
\]
This completes the proof.
\end{proof}

\begin{example}
The toric variety associated to the root system of type $A_n$ is isomorphic to $X(\Delta(B))$
for the building set $B=2^{\{1, \ldots, n+1\}} \setminus \{\emptyset\}$.
On the other hand, the toric varieties associated to the root systems of type $B$ and $C$
cannot be obtained from building sets since $a_{ij}=-2$ for some $i, j$.
\end{example}

Finally, we characterize root systems
whose associated toric varieties are Fano or weak Fano.

\begin{theorem}\label{ys:theorem9}
Let $\Phi$ be a root system.
Then the associated toric variety $X(\Delta(\Phi))$ is Fano (resp.\ weak Fano)
if and only if each irreducible component of $\Phi$ is of type $A_1$ or $A_2$
(resp. type $A_n$ or $B_n$ for some positive integer $n$).
\end{theorem}

\begin{proof}
Suppose that $\Phi$ is an irreducible root system of rank $n$.
We fix a set of simple roots $S=\{\alpha_1, \ldots, \alpha_n\} \subset \Phi$
and we put $a_{ij}=2(\alpha_i, \alpha_j)/(\alpha_j, \alpha_j)$.
For any $\tau \in \Delta(\Phi)(n-1)$,
there exists a set of simple roots $S' \subset \Phi$
such that $\tau$ is a face of $\sigma_{S'}$.
Proposition \ref{ys:Cartan} gives $(-K_{X(\Delta(\Phi))}.V(\tau))=\sum_{i=1}^na_{ij}$
for some $j=1, \ldots, n$.
Conversely, the proof of Proposition \ref{ys:Cartan} shows that
for any $j=1, \ldots, n$, there exists $\tau \in \Delta(\Phi)(n-1)$
such that $(-K_{X(\Delta(\Phi))}.V(\tau))=\sum_{i=1}^na_{ij}$.
Therefore by Proposition \ref{ys:Nakai}, $X(\Delta(\Phi))$ is Fano (resp.\ weak Fano)
if and only if $\sum_{i=1}^na_{ij}>0$ (resp.\ $\sum_{i=1}^na_{ij} \geq 0$)
for any $j=1, \ldots, n$.
Since $\Phi$ is irreducible and $(a_{ij})$ is the Cartan matrix,
$\sum_{i=1}^na_{ij}>0$ (resp.\ $\sum_{i=1}^na_{ij} \geq 0$)
for any $j=1, \ldots, n$ if and only if $\Phi$ is of type $A_1$ or $A_2$
(resp.\ type $A_n$ or $B_n$ for some positive integer $n$).
The theorem follows from the fact that the direct sum of root systems corresponds
to the product of the associated toric varieties.
\end{proof}

\section*{Acknowledgments}

The author would like to thank the organizers of the conference
``Algebraic and Geometric Combinatorics on Lattice Polytopes",
Takayuki Hibi and Akiyoshi Tsuchiya,
for their invitation and their kind hospitality.
Professor Hidefumi Ohsugi gave him valuable comments.
This work was supported by Grant-in-Aid for JSPS Fellows 18J00022.

\end{document}